\documentclass[12pt]{amsart}
\usepackage[dvips]{graphicx}
\usepackage{amssymb,amsmath,amsthm,cases}

\def\NZQ{\mathbb}               
\def\NN{{\NZQ N}}

\def\ZZ{{\NZQ Z}}

\def\frk{\mathfrak}               

\def\Phi{{\frk N}}
\def\Pc{{\mathcal P}}

\def\opn#1#2{\def#1{\operatorname{#2}}} 
\opn\chara{char} \opn\length{\ell} \opn\pd{pd} \opn\rk{rk}
\opn\projdim{proj\,dim} \opn\injdim{inj\,dim} \opn\rank{rank}
\opn\depth{depth} \opn\grade{grade} \opn\height{height}
\opn\size{size}
\opn\embdim{emb\,dim} \opn\codim{codim}

\opn\Tr{Tr} \opn\bigrank{big\,rank}
\opn\superheight{superheight}\opn\lcm{lcm}
\opn\trdeg{tr\,deg}
\opn\reg{reg} \opn\lreg{lreg} \opn\ini{in} \opn\lpd{lpd}
\opn\size{size}\opn{\mult}{mult}
\opn{\Cl}{Cl}
\opn\div{div} \opn\Div{Div} \opn\cl{cl} \opn\Cl{Cl}
\opn\Spec{Spec} \opn\Supp{Supp} \opn\supp{supp} \opn\Sing{Sing}
\opn\Ass{Ass} \opn\Min{Min} \opn\cl{cl}
\opn\Ann{Ann} \opn\Rad{Rad} \opn\Soc{Soc}
\opn\Syz{Syz} \opn\Im{Im} \opn\Ker{Ker} \opn\Coker{Coker}
\opn\Am{Am} \opn\Hom{Hom} \opn\Tor{Tor} \opn\Ext{Ext}
\opn\End{End} \opn\Aut{Aut} \opn\id{id} \opn\ini{in}

\opn\nat{nat}
\opn\pff{pf}
\opn\Pf{Pf} \opn\GL{GL} \opn\SL{SL} \opn\mod{mod} \opn\ord{ord}
\opn\Gin{Gin}
\opn\Hilb{Hilb}\opn\adeg{adeg}\opn\std{std}\opn\ip{infpt}
\opn\Pol{Pol}
\opn\sat{sat}
\opn\Var{Var}
\opn\Gen{Gen}
\opn\lex{lex}
\opn\div{div}

\opn\aff{aff} \opn\con{conv} \opn\relint{relint} \opn\st{st}
\opn\lk{lk} \opn\cn{cn} \opn\core{core} \opn\vol{vol}
\opn\link{link} \opn\star{star}
\opn\gr{gr}

\def\pot#1#2{#1[\kern-0.28ex[#2]\kern-0.28ex]}

\opn\dirlim{\underrightarrow{\lim}}
\opn\inivlim{\underleftarrow{\lim}}
\def\Implies{\ifmmode\Longrightarrow \else
        \unskip${}\Longrightarrow{}$\ignorespaces\fi}
\def\implies{\ifmmode\Rightarrow \else
        \unskip${}\Rightarrow{}$\ignorespaces\fi}
\def\iff{\ifmmode\Longleftrightarrow \else
        \unskip${}\Longleftrightarrow{}$\ignorespaces\fi}

\let\:=\colon

\def\NZQ{\mathbb}        
\def\NN{{\NZQ N}}

\def\ZZ{{\NZQ Z}}

\def\frk{\mathfrak}        

\def\Phi{{\frk N}}

\def \P{{\mathcal P}}

\newtheorem{Theorem}{Theorem}[section]

\theoremstyle{definition}

\newtheorem{Example}[Theorem]{Example}

\begin{document}
\title{Toric rings of nonsimple polyominoes}
\author {Akihiro Shikama}

\thanks{}


\address{Akihiro Shikama, Department of Pure and Applied Mathematics, Graduate School of Information Science and Technology,
Osaka University, Toyonaka, Osaka 560-0043, Japan}
\email{a-shikama@cr.math.sci.osaka-u.ac.jp}

\maketitle
\begin{abstract}
It is known that toric ring of a simple polyomino is ring homomorphic to a edge ring of a weakly chordal bipartite graph.  
In this paper we identify the attached toric rings of nonsimple polyominoes which are of the form ``rectangle minus rectangle''.
\end{abstract}

\section*{Introduction}
Polyominoes are two dimensional objects which are originally 
rooted in recreational mathematics and combinatorics. They have been widely discussed in connection with tiling problems of the plane.
Typically, a polyomino is plane figure obtained by joining squares of equal sizes, which are known as cells. In connection with commutative algebra, polyominoes are first discussed in \cite{Q} by assigning each polyomino the ideal of inner 2-minors or the {\it polyomino ideal}.
The study of ideal of $t$-minors of an  $m \times n$
matrix is a classical subject in commutative algebra.
The class of polyomino ideal widely generalizes the class of ideals of 2-minors of $m \times n$ matrix as well ass the ideal of inner 2-minors attached to a two sided ladder.

Let $\P$ be a polyomino and $K$ be a field.
We denote by $I_\P$, the polyomino ideal attatched to $\P$,
in a suitable polynomial ring over $K$.
The residue class ring defined by $I_\P$ denoted by $K[\P]$.
It is natural to investigate the algebraic properties of $I_\P$
depending on shape of $\P$.
The classes of polyominoes whose polyomino ideal is 
prime has been discussed in many papers.
In \cite{Q} it is shown that polyomino ideal of convex polyominoes are prime. 
In \cite{HQS} they introduced the class of {\em balanced}
polyominoes and proved that polyomino ideal of balanced polyominoes are prime.
Later in \cite{HM} it is shown that simple polyominoes are balanced.
In \cite{QSS}, independently from \cite{HM}, 
it is proved that polyomino ideal of simple polyominoes are prime
by identifying them with toric rings of edge rings of graphs.  
Very recently, in \cite{HQ} it is shown that polyomino ideal of polyomino of the shape ``rectangle minus convex'' is prime
by using localization argument on the attached rings of those polyominoses.

In this paper, we identify the toric rings of polyomino ideals for the special class of nonsimple polyominoes, ``rectangle minus rectangle''.

\section{Toric rings of nonsimple polyominoes}
First we recall some definitions and notation from \cite{Q}.  Given $a=(i,j)$ and $b=(k,l)$ in $\NN^2$ we write  $a\leq b$ if $i\leq k$ and $j\leq l$. The set $[a,b]=\{c\in\NN^2\:\; a\leq c\leq b\}$ is called an {\em interval}. If $i <k$ and $j<l$, then the elements $a$ and $b$ are called {\em diagonal} corners and $(i,l)$ and $(k,j)$ are called {\em anti-diagonal}
 corners of $[a,b]$. An interval of the from $C=[a,a+(1,1)]$ is called a {\em cell} (with left lower corner $a$). The elements (corners) $a, a+(0,1), a+(1,0), a+(1,1)$ of $[a,a+(1,1)]$ are called the {\em vertices} of $C$. The sets $\{a,a+(1,0)\}, \{a,a+(0,1)\},  \{a+(1,0),  a+(1,1)\}$ and   $\{a+(0,1),  a+(1,1)\}$ are called the {\em edges} of $C$. We denote the set of edge of $C$ by $E(C)$.

Let $\Pc$ be a finite collection of cells of $\NN^2$. The the vertex set of $\Pc$, denoted by $V(\Pc)$ is given by $V(\Pc)=\bigcup_{C \in \Pc} V(C)$.  
For each vertex $v$ in $V(\P)$, we write $v = (v_1,v_2)$ where $v_1$ is the first and $v_2$ is the second coodinate of $v$.
The edge set of $\Pc$, denoted by $E(\Pc)$ is given by $E(\Pc)=\bigcup_{C \in \Pc} E(C)$. Let $C$ and $D$ be two cells of $\Pc$. Then $C$ and $D$ are said to be {\em connected}, if there is a sequence of cells $\mathcal{C}:C= C_1, \ldots, C_m =D$ of $\Pc$  such that $C_i \cap C_{i+1}$ is an edge of $C_i$ for $i=1, \ldots, m-1$. If in addition, $C_i \neq C_j$ for all $i \neq j$, then $\mathcal{C}$ is called a {\em path} (connecting $C$ and $D$). The collection of cells $\Pc$ is called a {\em polyomino} if any two cells of  $\Pc$ are connected, see Figure~\ref{polyomino}.

 \begin{figure}[htbp]
\includegraphics[width = 3cm]{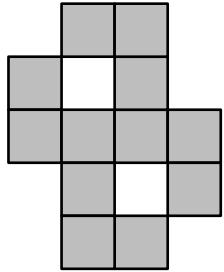}
\caption{polyomino}\label{polyomino}
\end{figure}

Let $\Pc$ be a polyomino, and let $K$ be a field. We denote by $S$ the polynomial ring over $K$ with variables $x_v$ with $v\in V(\Pc)$. We write for $v = (i,j)$, we may write $x_v = x_{ij}$ if needed.
Following \cite{Q} a $2$-minor $x_{ij}x_{kl}-x_{il}x_{kj}\in S$ is called an {\em inner minor} of $\Pc$ if all the cells $[(r,s),(r+1,s+1)]$ with $i\leq r\leq k-1$ and $j\leq s\leq l-1$ belong to $\Pc$. In that case the interval $[(i,j),(k,l)]$ is called an {\em inner interval} of $\Pc$. The ideal $I_\Pc\subset S$ generated by all inner minors of $\Pc$ is called the {\em polyomino ideal} of $\Pc$. We also set $K[\Pc]=S/I_\Pc$.



For each interval $[a,b]$ we regard the polyomino $\P_{(a,b)}$ in the obvious way.
We give Theorem~\ref{boxbox} that identifies a toric ring for
some specified cases of nonsimple polyominoes which is obtained by subtracting a rectangle from a bigger rectangle.

Hereafter let $\P$ be a  polyomino with $\P = \P_{(a,b)} \setminus \P_{(a'b')}$ with $(a<a'<b'<b)$. 
We define a map $\mu : V(\P) \rightarrow \ZZ$
as follows.
For $v =(v_1,v_2)$ in $\P$,
\[
\mu(v) = \begin{cases}
1~(for~a_1 \le v_1 \le a'_1, a_2 \le  v_2  \le a'_2 ) \\
1~(for~b'_1 \le v_1\le b_1, b'_2 \le v_2 \le b_2)\\
2~(for~a_1 \le v_1 \le a'_1,b'_2 \le v_2 \le b_2)\\
2~(for~b'_1 \le v_1 \le b_1, a_2 \le v_2 \le a'_2)\\
v_1 - a'_1 + 2 ~(for~a'_1< v_1 < b'_1, a_2 \le v_2\le b'_2) \\
v_2 - a'_2 + (b'_1 - a'_1) + 1 \\
~~(for~b'_1 \le v_1 \le b_1, a'_2 < v_2 < b'_2)\\
b'_1- v_1 + (b'_1 -a'_1) + (b'_2- a'_2) \\
~~~(for~a'_1 < v_1 < b'_1, b'_2 \le v_2 \le b_2)\\
b'_2 - v_2 + 2(b'_1-a'_1) + (b'_2-a'_2) -1\\
~~~(for~a_1 \le v_1 \le a'_1, a'_2 < v_2 < b'_2).
\end{cases}
\]
We call $\mu(v)$ the {\em labelling} of $v$.
\begin{Example}\label{ex}
for $a =(1,1)$ $b = (7,5)$, $a' = (2,2)$ and $b' = (5,4)$,
the labelling of vertices are given as follows.
\begin{figure}[htbp]
\includegraphics[width = 4.5cm]{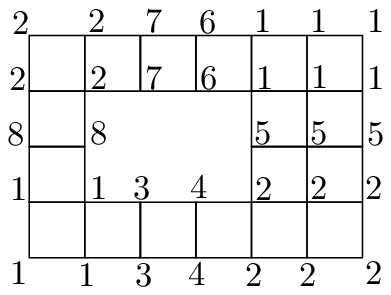}
\caption{labelling of $\P$}\label{label}
\end{figure}

We put 1 or 2 for each corners and
$3,4,\ldots $ for the other vertices anti-clockwisely.  
\end{Example}

Let $T = K[r_{v_1}s_{v_2} t_{\mu(v)} \mid v \in  V(\P) ]$ be the subring of a polynomial ring \\
$K[r_{a_1} ,\ldots r_{b_1}, 
s_{a_2}, \ldots s_{b_2}, t_1 \ldots t_M ]
$
where $M = \max \{\mu(v) \mid v \in V(\P)\}.$
We define the surjective ring homomorphism 
$\varphi : K[\P] \rightarrow T$
by setting $\varphi(x_v) = r_{v_1} s_{v_2}t_{\mu(v)}$.
For the sake of convenience, we write $\varphi(v) = r_{v_1}s_{v_2}t_{\mu(v)} = r_vs_vt_v$ in case it does not make readers confused.
We call the {\em toric ideal} $J_\P = {\text ker} \varphi$

\begin{Theorem}\label{boxbox}
Let $\P = \P_{(a,b)} \setminus \P_{(a'b')}$ with $(a<a'<b'<b)$. 
Then $I_\P = J_\P$. 
\end{Theorem}
\begin{proof}

First it is easy to see that $I_\P \subset$ ker$\varphi$.
We show that every quadratic binomials belong to $J_\P$ is a generator of $I_\P$, namely an inner 2-minor in $\P$ and every binomials of higher degree belongs to $I_\P$.

First, suppose that $ x_{v_1}x_{v_2} -x_{v_3}x_{v_4}$ is a quadratic binomial belongs to $J_\P$ with $x_{v_1}x_{v_2}\neq  x_{v_3}x_{v_4}$. 
Then either $(r_{v_1}, r_{v_2}, s_{v_1} s_{v_2}) = (r_{v_3}, r_{v_4}, s_{v_4} s_{v_3})$ or $(r_{v_1}, r_{v_2}, s_{v_1} s_{v_2}) = (r_{v_4}, r_{v_3}, s_{v_3} s_{v_4})$ holds. 
Thus, the quadratic binomial $ x_{v_1}x_{v_2} -x_{v_3}x_{v_4} $ is associated with an interval.
Now we show that the interval is an inner interval in $\P$.
Assume that the interval spanned by $x_{v_1}x_{v_2}-x_{v_3}x_{v_4}$  has $v_1$ and $v_2$ as diagonal corners. 
Also, we may assume that$r_{v_1} = r_{v_3}$ and 
 $s_{v_1}=s_{v_4}$ hold.

Suppose $[v_1,v_2]$ is not an inner interval in $\P$. 
Then $\P_{[a',b']} \subset \P_{[v_1,v_2]}$ can not be happened because 
otherwise, from our definition of $\mu$, one has $t_1 ^2=t_{
v_1}t_{v_2}  \neq t_{v_3}t_{v_4} = t_2^2$.
Therefore we may assume that one of $[v_1,v_3]$, $[v_1,v_4]$, $[v_3, v_2]$ or $[v_4, v_2]$ is not an inner interval in $\P$. 
Then from our definition of labelling we can not have $t_{v_1}t_{v_2}  = t_{v_3}t_{v_4}$.
Hence every degree 2 binomial in $J_\P$ is a generator of $I_\P$.

Now we show that each binomial of higher degree belongs to $I_\P$.
Let $f = f^{(+)}- f^{(-)}  \in$ ker$\varphi$ be a homogeneous binomial. Let $V_+= {v_1}\ldots {v_s} $ be the set of vertices such that  $x_{v_i}$ appear in $f^{(+)}$ and  $V_-= {u_1}\ldots x_{u_t}$ the vertices appear in $f^{(-)}$. 

We may assume that $V_+ \cap V_- = \emptyset$, otherwise, say if $v_1 = u_1$ holds, we have $f = x_{v_1} (f^{(+)}/x_{v_1} - f^{(-)}/x_{u_1})$
and we can check $f' =  f^{(+)}/x_{v_1} - f^{(-)}/x_{u_1}$ instead of $f$ to complete the proof.



Now we show that if we have $v \in V_+\cup V_-$ such that $\mu(v) \notin \{1,2\}$, then we are done in this case.
Assume that we have a vertex $u_1$ in $V_-$ such that 
$u_{1_1} \le a'_1 $ and $\mu(v) \neq 1,2$. 
Since $f$ belongs to the kernel, we have 2 vertices, say 
$v_1$ and $v_2$ in $V_+$ such that $s_{v_1}t_{v_1}| \varphi (x_{u_1})$ and
$r_{v_1} | \varphi (x_{{u_1}}) $.
Let $c$ be the rest corner of the interval spanned by $v_1$, $v_2$ and $u_1$. Our rule of labelling show that the interval is an inner interval in $\P$
See Figure~\ref{fig:side} to see the situation.
\begin{figure}[htbp]
\includegraphics[width = 4.5cm]{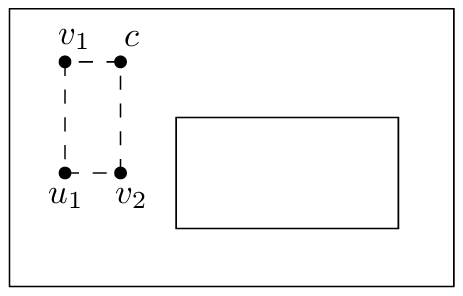}
\caption{}\label{fig:side}
\end{figure}
We obtain 
\begin{align*}\
f &= f^{(+)}-f^{(-)}\\
&= x_{v_1}x_{v_2}\frac{f^{(+)}}{x_{v_1}x_{v_2}}  -x_c\frac{{f^{(-)}}}{x_c}\\
&= (x_{v_1}x_{v_2}-x_{u_1}x_c)\frac{f^{(+)}}{x_{v_1}x_{v_2}}  + x_{u_1}x_c\frac{f^{(+)}}{x_{v_1}x_{v_2}}  -x_c\frac{{f^{(-)}}}{x_c}\\
&= (x_{v_1}x_{v_2}-x_{u_1}x_c)\frac{f^{(+)}}{x_{v_1}x_{v_2}}  + x_c\left( x_{u_1}\frac{f^{(+)}}{x_{v_1}x_{v_2}}  -\frac{{f^{(-)}}}{x_c} \right)
\end{align*}
Since $x_{v_1}x_{v_2}-x_{u_1}x_c$ is an inner minor of $\P$, we have
$x_{u_1}({f^{(+)}}/{x_{v_1}x_{v_2}})  - ({{f^{(-)}}}/{x_c}) \in J_\P$, 
we can apply induction on the degree of $f$ to complete the proof.

Now we assume every vertex in $V_{(+)} $ and $V_{(-)}$ are labelled 1 or 2.
For  $u_1  \in V_{(-)}$, we can find a vertex $v_1$ and $v_2$ in $V_{+} $ where $r_{u_1}  = r_{v_1}$ and $s_{u_1} = s_{v_2}$.
If $t_{u_1} = t_{v_1}$ or $t_{u_1} = t_{v_2}$, then using the same formula as above, we can apply induction.
Now we assume that $t_{u_1} = t_1$ and $t_{v_1} = t_{v_2} = t_2$.
By using similar argument, one has $u_2 \in V_-$ such that $s_{u_2} =s_{v_1}$.
Since $\varphi(f^{(+)}) = \varphi(f^{(-)})$ holds, we have a vertex  $u_2 \in V_-$  such that $t_{u_2} = t_1$.
See Figure~\ref{situation} for this situation.
\begin{figure}[htbp]
\includegraphics[width = 4.5cm]{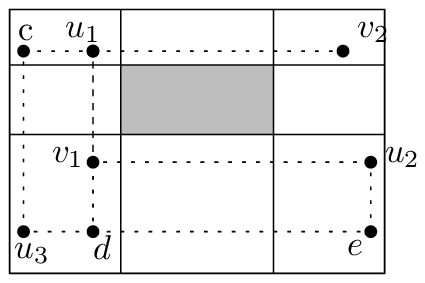}
\caption{}\label{situation}
\end{figure}
Then the interval spanned by $u_1$ and $u_3$  is an inner minor of $\P$.
Let $c,d$ be the other corners of this interval.
Then, 
\begin{align*}
f &= f^{(+)}- f^{(-)}\\
&= f^{(+)} - x_{u_1}x_{u_3}\frac{f^{(-)}}{x_{u_1}x_{u_3}}\\
&= f^{(+)}-x_cx_d\frac{f^{(-)}}{x_{u_1}x_{u_3}} - (x_{u_1}x_{u_3}-x_cx_d)\frac{f^{(-)}}{x_{u_1}x_{u_3}} 
\end{align*}
Notice that one of $t_{c} =t_1$ or $t_{d}=t_1$ holds.
Assume $t_{d} = t_1$. 
Then the interval spanned by $u_2$ and $d$ has $v_1$ as a corner vertex.
Suppose that $e$ is the rest corner vertex of this interval.
Then, we have
\begin{align*}
f =& f^{(+)}-x_cx_d\frac{f^{(-)}}{x_{u_1}x_{u_3}} - (x_{u_1}x_{u_3}-x_cx_d)\frac{f^{(-)}}{x_{u_1}x_{u_3}} \\
=& x_{v_1}\frac{f^{(+)}}{x_{v_1}}-x_cx_dx_{u_2}\frac{f^{(-)}}{x_{u_1}x_{u_2}x_{u_3}} - (x_{u_1}x_{u_3}-x_cx_d)\frac{f^{(-)}}{x_{u_1}x_{u_3}} \\
=& x_{v_1}\frac{f^{(+)}}{x_{v_1}}-x_{v_1}x_cx_e\frac{f^{(-)}}{x_{u_1}x_{u_2}x_{u_3}}-
\left(x_dx_{u_2}-x_{v_1}x_e\right)\left(x_c\frac{f^{(-)}}{x_{u_1}x_{u_2}x_{u_3}}\right)
\\
&- (x_{u_1}x_{u_3}-x_cx_d)\frac{f^{(-)}}{x_{u_1}x_{u_3}}\\
=& x_{v_1}\left(\frac{f^{(+)}}{x_{v_1}}-x_cx_e\frac{f^{(-)}}{x_{u_1}x_{u_2}x_{u_3}}\right)-
\left(x_dx_{u_2}-x_{v_1}x_e\right)\left(x_c\frac{f^{(-)}}{x_{u_1}x_{u_2}x_{u_3}}\right)
\\
&- (x_{u_1}x_{u_3}-x_cx_d)\frac{f^{(-)}}{x_{u_1}x_{u_3}}.
\end{align*}
Since $\left(x_dx_{u_2}-x_{v_1}x_e\right)$ and
$(x_{u_1}x_{u_3}-x_cx_d)$ are inner minors of $\P$ and since 
degree of ${f^{(+)}}/{x_{v_1}}-x_cx_e {f^{(-)}}/{x_{u_1}x_{u_2}x_{u_3}}$ 
is less than degree of $f$, we can apply induction on degree of $f$ to complete the proof.
\end{proof}

\begin{Example}
A toric ring of the polyomino given in Example~\ref{ex} is identified by Theorem \ref{boxbox}
as follows.

\[
S/I_\P \cong 
K 
\begin{bmatrix}
r_1s_5t_2&r_2s_5t_2&r_3s_5t_7&r_4s_5t_6&r_5s_5t_1&r_6s_5t_1&r_7s_5t_1\\
r_1s_4t_2&r_2s_4t_2&r_3s_4t_7&r_4s_4t_6&r_5s_4t_1&r_6s_4t_1&r_7s_4t_1\\
r_1s_3t_8&r_2s_3t_8&&&r_5s_3t_5&r_6s_3t_5&r_7s_3t_5\\
r_1s_2t_1&r_2s_2t_1&r_3s_2t_3&r_4s_2t_4&r_5s_2t_2&r_6s_3t_2&r_7s_2t_2\\
r_1s_1t_1&r_2s_1t_1&r_3s_1t_3&r_4s_1t_4&r_5s_1t_2&r_6s_2t_2&r_7s_1t_2\\
\end{bmatrix}.
\]
\end{Example}

In this paper we identify toric rings of polyominoes 
which are of the form ``rectangle minus rectangle''.
In general it is interesting but not so easy to identify 
toric ring for a given binomial prime ideal.
The wider class of polyominoes which can be discussed about
its toric ring is ``rectangle minus convex'',
because it is proved in \cite{HQ} that their polyomino ideals are prime. 


\end{document}